\documentclass[12pt]{article}

\usepackage{amsmath,amsthm,amssymb}

\newcommand{\N}{\mathbb{N}}
\newcommand{\Z}{\mathbb{Z}}
\newcommand{\Q}{\mathbb{Q}}

\newcommand{\sym}{\mathfrak{S}}
\newcommand{\Zmap}{\overline{\mathfrak{Z}}}
\newcommand{\h}{\mathfrak{H}}

\newcommand{\vece}{\vec{e}\,}
\newcommand{\vecj}{\vec{\jmath}\,}

\theoremstyle{theorem}
\newtheorem{theorem}{Theorem}[section]
\newtheorem{proposition}[theorem]{Proposition}
\newtheorem{lemma}[theorem]{Lemma}
\newtheorem{corollary}[theorem]{Corollary}
\newtheorem{conjecture}[theorem]{Conjecture}

\theoremstyle{definition}

\newtheorem{remark}[theorem]{Remark}

\title{On some multiple zeta-star values of one-two-three indices}
\author{Koji Tasaka and Shuji Yamamoto}
\date{}

\begin{document}

\maketitle

\begin{abstract}
In this paper, we present some identities for multiple zeta-star values 
with indices obtained by inserting 3 or 1 into the string $2,\ldots,2$. 
Our identities give analogues of Zagier's evaluation of 
$\zeta(2,\ldots,2,3,2,\ldots, 2)$ and examples of a kind of duality 
of multiple zeta-star values. 
Moreover, their generalizations give partial solutions of conjectures 
proposed by Imatomi, Tanaka, Wakabayashi and the first author. 
\end{abstract}

\section{Introduction}
Multiple zeta values and multiple zeta-star values 
(MZVs and MZSVs for short) are defined by the convergent series 
\begin{align*} 
\zeta(k_1,k_2,\ldots,k_n)&=\sum_{m_1>m_2>\cdots>m_n>0}
\frac{1}{m_1^{k_1}m_2^{k_2}\cdots m_n^{k_n}}, \\ 
\zeta^\star(k_1,k_2,\ldots,k_n)&=\sum_{m_1\geq m_2\geq\cdots\geq m_n\geq 1}
\frac{1}{m_1^{k_1}m_2^{k_2}\cdots m_n^{k_n}}, 
\end{align*}
where $k_1,k_2,\ldots,k_n$ are positive integers with $k_1\geq 2$. 
The weight and depth of the above series are by definition 
the integers $k=k_1+\cdots+k_n$ and $n$, respectively. 

We are interested in $\Q$-linear relations among these real values. 
This topic has been studied by many mathematicians and physicists. 
In this paper, we prove the following new $\Q$-linear relations among MZSVs, 
which are conjectured by M.~Kaneko in his unpublished work: 

\begin{theorem}\label{thm:main} 
For any positive integers $n$ and $m$, we have
\begin{align}
\label{eq:main1} 
&\zeta^\star(\{2\}^m,1)\cdot\zeta^\star(\{ 2 \}^n,1)
=\zeta^\star(\{2\}^m,1,\{2\}^n,1)+\zeta^\star(\{2\}^n,1,\{2\}^m,1), \\
\label{eq:main2} 
&\zeta^\star(\{2\}^m,1)\cdot\zeta^\star(\{2\}^n)
=\zeta^\star(\{2\}^m,1,\{2\}^n)+\zeta^\star(\{2\}^{n-1},3,\{2\}^m), \\
\label{eq:main3} 
&\zeta^\star(\{2\}^m)\cdot\zeta^\star(\{2\}^n)
=\zeta^\star(\{2\}^{m-1},3,\{2\}^{n-1},1)
+\zeta^\star(\{2\}^{n-1},3,\{2\}^{m-1},1), 
\end{align}
where $\{2\}^n$ stands for the $n$-tuple of $2$.
\end{theorem}

The identity \eqref{eq:main2} will be shown in Section 2, 
and \eqref{eq:main1} and \eqref{eq:main3} in Section 3. 
In this introduction, we make some remarks on these formulas. 

\bigskip 

First, we note that the identity \eqref{eq:main1} is already shown 
by Ohno and Zudilin (\cite{OZ}) using their `two-one formula' 
\[\zeta^\star(\{2\}^m,1,\{2\}^n,1)=4\zeta^\star (2m+1,2n+1)-2\zeta(2m+2n+2)\]
of depth $2$. 
Our proof of \eqref{eq:main1} is, however, fairly simpler than theirs 
and also works for proving \eqref{eq:main3} almost identically. 
Moreover, the same method leads to the following generalizations 
of \eqref{eq:main1} and \eqref{eq:main3}. 

\begin{theorem}\label{thm:1,3ext}
Let $n>0$ be an integer. 
\begin{enumerate}
\item[\upshape (i)] 
For any non-negative integers $j_1,j_2,\ldots,j_n$ such that 
$j_1,j_n\geq 1$, we have
\begin{equation}\label{eq:1ext} 
\sum_{k=0}^n (-1)^k \zeta^\star(\{2\}^{j_1},1,\ldots,\{2\}^{j_k},1)\cdot
\zeta^\star(\{2\}^{j_n},1,\ldots,\{2\}^{j_{k+1}},1)=0. 
\end{equation}
\item[\upshape (ii)] 
For any non-negative integers $j_1,j_2,\ldots, j_{2n}$, we have
\begin{equation}\label{eq:3ext} 
\begin{split}
&\sum_{k=0}^n 
\zeta^\star(\{2\}^{j_1},3,\{2\}^{j_2},1,\ldots,3,\{2\}^{j_{2k}},1)\cdot
\zeta^\star(\{2\}^{j_{2n}},3,\{2\}^{j_{2n-1}},1,\ldots,3,\{2\}^{j_{2k+1}},1)\\
&=\sum_{k=1}^n \zeta^\star(\{2\}^{j_1},3,\{2\}^{j_2},1,\ldots,
1,\{2\}^{j_{2k-1}+1})\cdot
\zeta^\star(\{2\}^{j_{2n}},3,\{2\}^{j_{2n-1}},1,\ldots,1,\{2\}^{j_{2k+1}+1}).
\end{split}
\end{equation}
\end{enumerate}
\end{theorem}

In fact, we prove even more general formula than Theorem \ref{thm:1,3ext} 
(see Theorem \ref{thm:sec3main}). 

The identity \eqref{eq:3ext} has an application to the following conjecture: 

\begin{conjecture}[{\cite[Conjecture 4.5]{TWIT}}]\label{conj:twit}
\begin{enumerate}
\item[$(A)$] 
Let $n$ be a positive integer, and 
$j_0,j_1,\ldots,j_{2n-1}$ non-negative integers. 
Put $m=j_0+j_1+\cdots+j_{2n-1}$. Then we have
\[\sum_{\sigma\in\sym_{2n}} 
\zeta^\star(\{2\}^{j_{\sigma(0)}},3,\{2\}^{j_{\sigma(1)}},1,
\{2\}^{j_{\sigma(2)}},\ldots,3,\{2\}^{j_{\sigma(2n-1)}},1)
\overset{?}{\in} \Q\cdot\pi^{2m+4n}.\]
\item[$(B)$] 
Let $n, j_0,j_1,\ldots,j_{2n}$ be non-negative integers. 
Put $m=j_0+j_1+\cdots+j_{2n}$. Then we have
\[\sum_{\sigma\in\sym_{2n+1}} 
\zeta^\star(\{2\}^{j_{\sigma(0)}},3,\{2\}^{j_{\sigma(1)}},1,
\{2 \}^{j_{\sigma(2)}},\ldots,3,\{2 \}^{j_{\sigma(2n-1)}},1,
\{2\}^{j_{\sigma(2n)}+1})
\overset{?}{\in} \Q\cdot\pi^{2m+4n+2}.\]
\end{enumerate}
\end{conjecture}

We denote by $(A_n)$ (resp.\ $(B_n)$) 
the statement $(A)$ (resp.\ $(B)$) in Conjecture \ref{conj:twit} 
for a specific value of $n$. For example, $(B_0)$ means that 
$\zeta^\star(\{2\}^{j+1})\in\Q\cdot\pi^{2j+2}$ for $j\geq 0$, 
which is already known (see \cite{Zl}). 
On the other hand, $(A_1)$ is a consequence of $(B_0)$ and 
the identity \eqref{eq:main3} of Theorem \ref{thm:main}. 
This implication $(B_0)\implies(A_1)$ is generalized as follows: 

\begin{theorem}\label{thm:TWIT A_n}
Let $n$ be a positive integer. 
If the statements $(A_l)$ and $(B_l)$ hold for all $l<n$, 
then $(A_n)$ is also true. 
\end{theorem}

\bigskip

Next, we explain two topics related with the identity \eqref{eq:main2}. 
The first is on the `duality' for MZSVs. 
The first example of such phenomena was found by Kaneko-Ohno \cite{KO}, 
who proved that 
\begin{equation}\label{eq:KOduality}
(-1)^{n+1}\zeta^\star(m+1,\{1\}^n)-(-1)^{m+1}\zeta^\star(n+1,\{1\}^m)
\in\Q[\zeta(2),\zeta(3),\zeta(5),\ldots]. 
\end{equation}
This is regarded as an analogue of the duality of MZVs 
\[\zeta(m+1,\{1\}^{n-1})-\zeta(n+1,\{1\}^{m-1})=0. \]
They also formulated a conjecture which generalizes \eqref{eq:KOduality}, 
and recently Li \cite{Li} and Yamazaki \cite{Ya} proved this conjecture 
using generalized hypergeometric functions. 

On the other hand, since $\zeta^\star(\{2\}^m,1)=2\zeta(2m+1)$ and 
$\zeta^\star(\{2\}^n)=2(1-2^{1-2n})\zeta(2n)$, 
our formula \eqref{eq:main2} implies the following: 

\begin{corollary}
For integers $m,n\geq 1$, we have 
\begin{equation}\label{eq:duality star}
(-1)^{m+n+1}\zeta^\star(\{2\}^m,1,\{2\}^n)
-(-1)^{m+n}\zeta^\star(\{2\}^{n-1},3,\{2\}^m)
\in\Q[\zeta(2),\zeta(3),\zeta(5),\ldots]. 
\end{equation}
\end{corollary}

This might be regarded as an analogue of the duality 
\begin{equation}\label{eq:duality}
\zeta(\{2\}^m,1\{2\}^n)-\zeta(\{2\}^n,3,\{2\}^{m-1})=0. 
\end{equation}
We remark that \eqref{eq:duality star} is not contained 
in the Kaneko-Ohno conjecture. 
This suggests that the `duality' for MZSVs might be open to further extension, 
but we will not develop this point in the present paper. 

The second topic is inspired by Zagier's work \cite{Z} on an evaluation of 
$\zeta(\{2\}^m,3,\{2 \}^n)$ which plays an important role 
in Brown's partial settlement of Hoffman's basis conjecture for MZVs 
(\cite{B1}, see also \cite{B2}). 
In fact, Zagier also proved an analogous evaluation of MZSV: 
\begin{equation}\label{eq:22322}
\begin{split}
\zeta^\star&(\{2\}^m,3,\{2 \}^n)\\
&=-2\sum_{r=1}^{m+n+1}\Biggl(\binom{2r}{2n}-\delta_{r,n}
-(1-4^{-r})\binom{2r}{2m+1}\Biggr)
\,\zeta(2r+1)\,\zeta^\star(\{2 \}^{m+n+1-r}). 
\end{split}
\end{equation}
Note that, since the duality \eqref{eq:duality} holds, 
the formula for $\zeta(\{2\}^m,3,\{2 \}^n)$ 
may also be regarded as an evaluation of $\zeta(\{2\}^{m+1},1,\{2\}^n)$. 
On the other hand, an evaluation for $\zeta^\star(\{2\}^{m+1},1,\{2\}^{n})$ 
can be obtained by combining \eqref{eq:22322} 
with our result \eqref{eq:main2}: 

\begin{theorem}\label{thm:22122} 
For any non-negative integers $m,n$, we have 
\begin{equation}\label{eq:22122}
\begin{split}
\zeta^\star&(\{2\}^{m+1},1,\{2\}^{n})\\
&=2\sum_{r=1}^{m+n+1}\Biggl(\binom{2r}{2m+2}-(1-4^{-r})\binom{2r}{2n-1}\Biggr)
\,\zeta(2r+1)\,\zeta^\star(\{2\}^{m+n+1-r}). 
\end{split}
\end{equation}
\end{theorem}

We also remark that Theorem \ref{thm:22122} implies the following corollary, 
in the same way that \eqref{eq:22322} implies 
the corresponding result \cite[Theorem 2]{Z}: 

\begin{corollary}\label{cor:22122} 
For each odd integer $k\geq 3$, the $\Q$-vector space spanned by 
$\zeta^\star(\{2\}^{m+1},1,\{2\}^n)$ with $2m+2n+3=k$ and $m,n\geq 0$ is 
equal to the $\Q$-vector space spanned by 
$\pi^{2r}\zeta(k-2r)$ ($r=0,1,\ldots,(k-3)/2$). 
\end{corollary}

\section{On the identity \eqref{eq:main2}}
In this section, we prove the identity \eqref{eq:main2} 
in Theorem \ref{thm:main}. 

For an integer $p\geq 0$, we denote by $\zeta_p (k_1,\ldots,k_n)$ 
(resp.\ $\zeta^\star_p (k_1,\ldots,k_n)$) 
the finite sum obtained by truncating the series for 
$\zeta (k_1,\ldots,k_n)$ (resp.\ $\zeta^\star (k_1,\ldots,k_n)$):
\begin{align*} 
\zeta_p (k_1,k_2,\ldots,k_n)
&=\sum_{p\geq p_1>p_2>\cdots>p_n>0} 
\frac{1}{p_1^{k_1}p_2^{k_2}\cdots p_n^{k_n}}\\ 
\Biggl(\text{resp.\ } 
\zeta^\star_p (k_1,k_2,\ldots,k_n)
&=\sum_{p\geq p_1\geq p_2\geq\cdots\geq p_n\geq 1} 
\frac{1}{p_1^{k_1}p_2^{k_2}\cdots p_n^{k_n}} \Biggr).
\end{align*}
The empty sum is interpreted as $0$, 
and for the unique index $\varnothing$ of depth $0$, 
we put $\zeta_p(\varnothing)=\zeta^\star_p(\varnothing)=1$ even if $p=0$. 

Let $a$, $b$ and $c$ be positive integers. 
Our strategy for proving the identity \eqref{eq:main2} is 
to study the four generating functions 
\begin{alignat*}{2}
F_p(x,y)&=\sum_{m,n\geq 0}\zeta_p(\{a\}^m,b,\{c\}^n)x^my^n,\quad & 
G_p(y)&=\sum_{n\geq0}\zeta_p(\{c\}^n)y^n,\\
F^\star_p(x,y)&=\sum_{m,n\geq 0}\zeta^\star_p(\{c\}^m,b,\{a\}^n)x^my^n,\quad & 
G^\star_p(y)&=\sum_{n\geq0}\zeta^\star_p(\{a\}^n)y^n.
\end{alignat*}
Note that $F_0(x,y)=F^\star_0(x,y)=0$ and $G_0(y)=G^\star_0(y)=1$. 

\begin{lemma}\label{lem:FG} 
For any integer $p\geq 0$, we have
\begin{align}
\label{eq:FG} 
\begin{pmatrix}F_p(x,y)\\ G_p(y)\end{pmatrix}
&=T_pT_{p-1}\cdots T_1\begin{pmatrix}0\\ 1\end{pmatrix}, \\
\label{eq:FGstar} 
\begin{pmatrix}F^\star_p(x,y)\\ G^\star_p(y)\end{pmatrix}
&=U_pU_{p-1}\cdots U_1\begin{pmatrix}0\\1\end{pmatrix},
\end{align}
where 
\begin{align*}
T_q=T_q(x,y)
&=\begin{pmatrix}1+\frac{x}{q^a} & \frac{1}{q^b} \\ 
0 & 1+\frac{y}{q^c}\end{pmatrix}, \\
U_q=U_q(x,y)
&=\biggl(1-\frac{x}{q^c}\biggr)^{-1}\biggl(1-\frac{y}{q^a}\biggr)^{-1}
\begin{pmatrix}1-\frac{y}{q^a} & \frac{1}{q^b} \\ 
0 & 1-\frac{x}{q^c}\end{pmatrix}.
\end{align*}
\end{lemma}

\begin{proof}
The case $p=0$ is obvious. 
For $p>0$, by definition, we have 
\begin{align*}
F_p(x,y)
&=\sum_{m,n\geq 0}\sum_{p\geq p_1>\cdots>p_{m+n+1}>0}
\frac{x^my^n}{p_1^a \cdots p_m^a p_{m+1}^b p_{m+2}^c \cdots p_{m+n+1}^c}\\
&=F_{p-1}(x,y)
+\sum_{m,n\geq 0}\sum_{p=p_1>\cdots>p_{m+n+1}>0}
\frac{x^my^n}{p_1^a \cdots p_m^a p_{m+1}^b p_{m+2}^c \cdots p_{m+n+1}^c}\\
&=F_{p-1}(x,y)
+\sum_{n\geq 0}\sum_{p=p_1>\cdots>p_{n+1}>0}
\frac{y^n}{p_1^b p_2^c \cdots p_{n+1}^c}\\
&\qquad +\sum_{m>0,n\geq 0}\sum_{p=p_1>\cdots>p_{m+n+1}>0}
\frac{x^my^n}{p_1^a \cdots p_m^a p_{m+1}^b p_{m+2}^c \cdots p_{m+n+1}^c} \\
&=F_{p-1}(x,y)+\frac{1}{p^b}G_{p-1}(y)+\frac{x}{p^a}F_{p-1}(x,y). 
\end{align*}
In a similar but simpler way, we also obtain 
$G_p(y)=G_{p-1}(y)+\frac{y}{p^c}G_{p-1}(y)$. 
Hence we have 
\[\begin{pmatrix}F_p(x,y)\\ G_p(y)\end{pmatrix}
=\begin{pmatrix}1+\frac{x}{p^a} & \frac{1}{p^b} \\ 
0 & 1+\frac{y}{p^c}\end{pmatrix}
\begin{pmatrix}F_{p-1}(x,y) \\ G_{p-1}(y)\end{pmatrix}, \]
and the identity \eqref{eq:FG} by induction. 

The case \eqref{eq:FGstar} can be shown in a similar way. 
In fact, one has 
\begin{align*}
F^\star_p(x,y)
&=F^\star_{p-1}(x,y)+\frac{1}{p^b}G^\star_p (y)+\frac{x}{p^c}F^\star_p(x,y),\\
G^\star_p(y)&=G^\star_{p-1}(y)+\frac{y}{p^a} G^\star_p(y), 
\end{align*}
that is 
\[\begin{pmatrix} 1-\frac{x}{p^c} & -\frac{1}{p^b} \\
0 & 1-\frac{y}{p^a} \end{pmatrix}
\begin{pmatrix} F^\star_p(x,y) \\ G^\star_p(y) \end{pmatrix}
=\begin{pmatrix}F^\star_{p-1} (x,y) \\ G^\star_{p-1} (y) \end{pmatrix}. \]
It is easy to see that the inverse of the $2\times 2$ matrix in the left hand side 
is equal to $U_p$. 
\end{proof}

\begin{proposition}\label{prop:ccbaa} 
For any non-negative integers $p$, $m$ and $n$, we have
\begin{equation}\label{eq:ccbaa}
\zeta^\star_p(\{c\}^m,b,\{a\}^n)
=\sum_{k=0}^m\sum_{l=0}^n(-1)^{k+l}\zeta_p(\{a\}^l,b,\{c\}^k)
\cdot\zeta^\star_p(\{c\}^{m-k})\cdot\zeta^\star_p(\{a\}^{n-l}).
\end{equation}
\end{proposition}

\begin{proof}
Lemma \ref{lem:FG} implies that 
\begin{equation}\label{eq:Fstar,F} 
F^\star_p(x,y)=F_p(-y,-x)\prod_{q=1}^p\biggl(1-\frac{x}{q^c}\biggr)^{-1}
\prod_{q=1}^p \biggl(1-\frac{y}{q^a}\biggr)^{-1}.
\end{equation}
Here we have
\[\prod_{q=1}^p\biggl(1-\frac{x}{q^c}\biggr)^{-1}
=\sum_{m\geq 0}\zeta^\star_p(\{c\}^m)x^m, \quad 
\prod_{q=1}^p \biggl(1-\frac{y}{q^a}\biggr)^{-1}
=\sum_{n\geq 0}\zeta^\star_p(\{a\}^n)y^n. \]
By comparing the coefficients of \eqref{eq:Fstar,F}, 
we obtain \eqref{eq:ccbaa}.
\end{proof}

\begin{remark}
Here we make a remark on an algebraic interpretation of 
the identity \eqref{eq:ccbaa}. 

First we recall the setup of harmonic algebra (see \cite{H} for details). 
Let $\h=\Q\langle x,y \rangle$ be the non-commutative polynomial algebra 
in two indeterminates $x,y$, and $\h^1$ its subalgebra $\Q+\h y$. 
For an integer $p\geq 0$, we define the $\Q$-linear maps 
$Z_p\colon\h^1\longrightarrow\Q$ and $Z_p^\star\colon\h^1\longrightarrow\Q$ by 
\begin{alignat*}{2}
Z_p(1)&=1 \text{\ and \ } 
&Z_p(z_{k_1}\cdots z_{k_n})&=\zeta_p(k_1,\ldots,k_n), \\
Z_p^\star(1)&=1 \text{\ and \ } 
&Z_p^\star (z_{k_1}\cdots z_{k_n})&=\zeta^\star_p(k_1,\ldots,k_n), 
\end{alignat*}
where $z_k=x^{k-1}y$  $(k=1,2,\ldots)$. 
Let $\gamma$ be the algebra automorphism on $\h$ characterized by 
$\gamma(x)=x$ and $\gamma(y)=x+y$, 
and define the $\Q$-linear transformation $d\colon\h^1\longrightarrow\h^1$ by 
\[d(1)=1 \text{ and } d(wy)=\gamma(w)y\] 
for any word $w\in\h$. Then one has $Z_p\circ d=Z_p^\star$. 
Moreover, we define the harmonic product $*$, 
a $\Q$-bilinear product on $\h^1$, inductively by 
\begin{gather*}
1*w=w*1=w,\\
z_kw*z_lw'=z_k(w*z_lw')+z_l(z_kw*w')+z_{k+l}(w*w'),
\end{gather*}
where $k,l\geq 1$ and $w,w'\in \h^1$. 
This product $*$ makes $\h^1$ a commutative $\Q$-algebra, 
and $Z_p\colon\h^1\longrightarrow\Q$ an algebra homomorphism 
for any $p\geq 0$. 

Now the algebraic counterpart to the identity \eqref{eq:ccbaa} is 
the following: 

\begin{proposition}\label{prop:ccbaa alg}
For any $m,n\geq 0$, we have 
\begin{equation}\label{eq:ccbaa alg}
d(z_c^mz_bz_a^n)=\sum_{k=0}^m\sum_{l=0}^n
(-1)^{k+l}z_a^lz_bz_c^k*d(z_c^{m-k})*d(z_a^{n-l}). 
\end{equation}
\end{proposition}

In fact, this proposition is equivalent to the identity \eqref{eq:ccbaa}, 
since the map from $\h^1$ to $\Q^\N$ that sends $w$ 
to $\bigl\{Z_p(w)\bigr\}_p$ is injective (see \cite[\S 3]{Y} for details). 
Alternatively, one can also prove \eqref{eq:ccbaa alg} directly 
(this last remark was communicated by Shingo Saito). 
\end{remark}

Now we return to the proof of the identity \eqref{eq:main2} of 
Theorem \ref{thm:main}. 

\begin{proof}[Proof of \eqref{eq:main2}]
By Proposition \ref{prop:ccbaa}, we have for any integers $m,n>0$ 
\begin{align*}
\zeta^\star_p(\{2\}^m,1,\{2\}^n)
&=\sum_{k=0}^m\sum_{l=0}^n(-1)^{k+l}\zeta_p(\{2\}^l,1,\{2\}^k)\,
\zeta^\star_p(\{2\}^{m-k})\,\zeta^\star_p(\{2\}^{n-l}), \\
\zeta^\star_p(\{2\}^{n-1},3,\{2\}^m)
&=\sum_{k=0}^m\sum_{l=0}^{n-1}(-1)^{k+l}\zeta_p(\{2\}^k,3,\{2\}^l)\,
\zeta^\star_p(\{2\}^{m-k})\,\zeta^\star_p(\{2\}^{n-1-l})\\
&=\sum_{k=0}^m\sum_{l=1}^n(-1)^{k+l+1}\zeta_p(\{2\}^k,3,\{2\}^{l-1})\,
\zeta_p^\star(\{2\}^{m-k})\,\zeta^\star_p(\{2\}^{n-l}). 
\end{align*}
Hence, we have
\begin{equation}\label{eq:22122+2322} 
\begin{split}
&\zeta^\star_p(\{2\}^m,1,\{2\}^n)+\zeta^\star_p(\{2\}^{n-1},3,\{2\}^m)\\
&=\sum_{k=0}^m(-1)^k\zeta_p(1,\{2\}^k)\,
\zeta^\star_p(\{2\}^{m-k})\,\zeta^\star_p(\{2\}^n)\\
&\quad +\sum_{k=0}^m\sum_{l=1}^n(-1)^{k+l}
\Bigl\{\zeta_p(\{2\}^l,1,\{2\}^k)-\zeta_p(\{2\}^k,3,\{2\}^{l-1})\Bigr\}
\zeta^\star_p(\{2\}^{m-k})\,\zeta^\star_p(\{2\}^{n-l}).
\end{split}
\end{equation}
By the duality \eqref{eq:duality},  
the second term in the right hand side of \eqref{eq:22122+2322} 
vanishes when $p\to\infty$. 
Finally, using the identity 
\[\zeta^\star_p(\{2\}^m,1)
=\sum_{k=0}^m(-1)^k\zeta_p(1,\{2\}^k)\,\zeta^\star_p(\{2\}^{m-k}), \]
which is a special case of Proposition \ref{prop:ccbaa}, 
we prove the identity \eqref{eq:main2} by letting $p\to\infty$.
\end{proof}

\begin{remark}
Although both sides of the identity \eqref{eq:main2} diverge when $m=0$, 
the above proof shows that for any non-negative integer $p$ we have 
\begin{equation}\label{eq:main2 m=0}
\begin{split}
&\zeta^\star_p(1,\{2\}^n)+\zeta^\star_p(\{2\}^{n-1},3)\\
&=\zeta^\star_p(1)\zeta^\star_p(\{2\}^n)+\sum_{l=1}^n(-1)^l
\Bigl\{\zeta_p (\{2\}^l,1)-\zeta_p (3,\{2\}^{l-1})\Bigr\}
\zeta^\star_p(\{2\}^{n-l}).
\end{split}
\end{equation}
On the other hand, the harmonic product relation implies that 
\[\zeta^\star_p(1)\zeta^\star_p(\{2\}^n)
=\sum_{l=0}^n\zeta^\star_p(\{2\}^l,1,\{2\}^{n-l})
-\sum_{l=0}^{n-1}\zeta^\star_p(\{2\}^l,3,\{2\}^{n-1-l}). \]
Substituting it into \eqref{eq:main2 m=0}, we obtain 
\begin{align*} 
\zeta^\star_p(\{2\}^{n-1},3)
&=\sum_{l=1}^n\zeta^\star_p(\{2\}^l,1,\{2\}^{n-l})
-\sum_{l=0}^{n-1}\zeta^\star_p(\{2\}^l,3,\{2\}^{n-1-l})\\
&\quad +\sum_{l=1}^n(-1)^l
\Bigl\{\zeta_p(\{2\}^l,1)-\zeta_p(3,\{2\}^{l-1})\Bigr\}
\zeta^\star_p(\{2\}^{n-l}).
\end{align*}
By letting $p\rightarrow \infty$, we get the following:

\begin{proposition}\label{prop:m=0} 
For any positive integer $n$, we have
\[\zeta^\star(\{2\}^{n-1},3)
=\sum_{l=1}^n\zeta^\star(\{2\}^l,1,\{2\}^{n-l})
-\sum_{l=0}^{n-1}\zeta^\star(\{2\}^l,3,\{2\}^{n-1-l}).\]
\end{proposition}

This proposition was shown by Ihara, Kajikawa, Ohno and Okuda 
\cite[eq.~(12)]{IKOO} using the derivation relation for MZSVs. 
\end{remark}

\section{On the identities \eqref{eq:main1} and \eqref{eq:main3}}
In this section, we prove a formula which includes 
the identities \eqref{eq:main1} and \eqref{eq:main3} as special cases. 

First we introduce some notation. 
For an integer $n\geq 1$, 
consider two vectors $\vecj=(j_1,\ldots,j_n)\in\Z_{\geq 0}^n$ and 
$\vece=(e_1,\ldots,e_{n-1})\in\{1,3\}^{n-1}$ 
(here $\{1,3\}^{n-1}$ does not mean the sequence $(1,3,\ldots,1,3)$, 
but the $(n-1)$-th power of the set $\{1,3\}$). 
We set 
\begin{equation}\label{eq:Zmap}
\Zmap_n(\vecj,\vece):=
\zeta^\star\bigl(\{2\}^{j_1},e_1,\{2\}^{j_2},e_2,\ldots,
e_{n-1},\{2\}^{j_n}\bigr). 
\end{equation}
For $n=0$, we simply put $\Zmap_0=1$. 
Note that the right hand side of \eqref{eq:Zmap} diverges if and only if 
$n\geq 2$, $j_1=0$ and $e_1=1$. 
If this is not the case, we say that $(\vecj,\vece)$ is an admissible pair. 

For $\vecj=(j_1,\ldots,j_n)$, we define the following operations: 
\begin{align*}
\vecj_+&:=(j_1,\ldots,j_n,0),& \vecj^+&:=(j_1,\ldots,j_{n-1},j_n+1), \\
\vecj'&:=(j_n,\ldots,j_1),& \vecj|_k&:=(j_1,\ldots,j_k) 
\quad (k=0,\ldots,n). 
\end{align*}
For example, we have 
$(\vecj'|_{n-k})_+=(j_n,\ldots,j_{k+1},0)$. 
We also apply similar operations to $\vece$ in obvious manners. 

Our main result in this section is the following formula: 

\begin{theorem}\label{thm:sec3main}
Let $n$ be a positive integer, 
$\vecj=(j_1,\ldots,j_n)\in\Z_{\geq 0}^n$ and 
$\vece=(e_1,\ldots,e_{n-1})\in\{1,3\}^{n-1}$. 
Assume that both $(\vecj,\vece)$ and $(\vecj',\vece')$ are admissible pairs. 
Put $e_0=e_n=1$, and 
\[X(k):=\begin{cases}
\Zmap_{k+1}((\vecj|_k)_+,\vece|_k)\cdot
\Zmap_{n-k+1}((\vecj'|_{n-k})_+,\vece'|_{n-k}) 
& (\text{if $e_k=1$}), \\
\Zmap_k((\vecj|_k)^+,\vece|_{k-1})\cdot
\Zmap_{n-k}((\vecj'|_{n-k})^+,\vece'|_{n-k-1}) 
& (\text{if $e_k=3$}) 
\end{cases}\] 
for $k=0,\ldots,n$ 
(here $\vece|_n$ and $\vece'|_n$ stand for 
$(e_1,\ldots,e_n)$ and $(e_{n-1},\ldots,e_0)$, respectively). 
Then we have 
\[\sum_{k=0}^n(-1)^k X(k)=0. \]
\end{theorem}

Let us examine some examples. 
First, we set $\vece=(1,1,\ldots,1)$. Then we have 
\begin{align*}
X(k)
&=\Zmap_{k+1}\bigl((j_1,\ldots,j_k,0),(1,\ldots,1)\bigr)\cdot
\Zmap_{n-k+1}\bigl((j_n,\ldots,j_{k+1},0),(1,\ldots,1)\bigr)\\
&=\zeta^\star(\{2\}^{j_1},1,\ldots,\{2\}^{j_k},1)\cdot
\zeta^\star(\{2\}^{j_n},1,\ldots,\{2\}^{j_{k+1}},1). 
\end{align*}
Hence Theorem \ref{thm:sec3main} implies the identity \eqref{eq:1ext} 
in Theorem \ref{thm:1,3ext} in this case. Similarly, 
for $\vece=(3,1,3,\ldots,1,3)\in\{1,3\}^{2n-1}$, we have 
\begin{align*}
X(k)=
&\zeta^\star(\{2\}^{j_1},3,\{2\}^{j_2},1,\ldots,3,\{2\}^{j_k},1)\\
&\times
\zeta^\star(\{2\}^{j_{2n}},3,\{2\}^{j_{2n-1}},1,\ldots,3,\{2\}^{j_{k+1}},1)
\end{align*}
if $k$ is even, and 
\begin{align*}
X(k)=
&\zeta^\star(\{2\}^{j_1},3,\{2\}^{j_2},1,\ldots,
3,\{2\}^{j_{k-1}},1,\{2\}^{j_k+1})\\
&\times
\zeta^\star(\{2\}^{j_{2n}},3,\{2\}^{j_{2n-1}},1,\ldots,
3,\{2\}^{j_{k+2}},1,\{2\}^{j_{k+1}+1})
\end{align*}
if $k$ is odd. 
Therefore, we obtain the identity \eqref{eq:3ext} in Theorem \ref{thm:1,3ext}.

We may also apply Theorem \ref{thm:sec3main} to other cases. 
For example, by setting $\vece=(3,3)$ or $\vece=(3,1)$, one obtains 
\begin{align*}
\zeta^\star&(\{2\}^{j_1},3,\{2\}^{j_2},3,\{2\}^{j_3},1)
+\zeta^\star(\{2\}^{j_1+1})\cdot\zeta^\star(\{2\}^{j_3},3,\{2\}^{j_2+1})\\
&=\zeta^\star(\{2\}^{j_1},3,\{2\}^{j_2+1})\cdot\zeta^\star(\{2\}^{j_3+1})
+\zeta^\star(\{2\}^{j_3},3,\{2\}^{j_2},3,\{2\}^{j_1},1) 
\end{align*}
or 
\begin{align*}
\zeta^\star&(\{2\}^{j_1},3,\{2\}^{j_2},1,\{2\}^{j_3},1)
+\zeta^\star(\{2\}^{j_1+1})\cdot\zeta^\star(\{2\}^{j_3},1,\{2\}^{j_2+1})\\
&=\zeta^\star(\{2\}^{j_1},3,\{2\}^{j_2},1)\cdot\zeta^\star(\{2\}^{j_3},1)
+\zeta^\star(\{2\}^{j_3},1,\{2\}^{j_2},3,\{2\}^{j_1},1), 
\end{align*}
respectively. 

\bigskip
Now we proceed to the proof of Theorem \ref{thm:sec3main}. 
For $\infty\geq A\geq B\geq 1$, we put 
\begin{gather*}
C_{-1}(A,B)=\delta_{A,B}A^2,\qquad
C_0(A,B)=1,\\
C_j(A,B)=\sum_{A\geq a_1\geq\cdots\geq a_j\geq B}
\frac{1}{a_1^2\cdots a_j^2}\qquad (j=1,2,3,\ldots). 
\end{gather*}
Then we have 
\begin{equation}\label{eq:C_induction}
C_j(A,B)=\sum_{p=B}^A\frac{1}{p^2}C_{j-1}(p,B)
=\sum_{p=B}^AC_{j-1}(A,p)\frac{1}{p^2}
\end{equation}
for any $j\geq 0$. 

\begin{lemma}\label{lem:C_duality}
For any $j\geq -1$ and $1\leq p,q\leq \infty$, we have 
\begin{equation}\label{eq:C_duality}
\sum_{p_0=1}^pC_j(p,p_0)\frac{q}{p_0(p_0+q)}
=\sum_{q_0=1}^qC_j(q,q_0)\frac{p}{q_0(q_0+p)}. 
\end{equation}
Here, $\frac{p}{q_0(q_0+p)}$ means $\frac{1}{q_0}$ if $p=\infty$, 
and so on. 
\end{lemma}
\begin{proof}
We use induction on $j$. 
When $j=-1$, both sides are equal to $\frac{pq}{p+q}$. 
Thus we assume $j\geq 0$. By \eqref{eq:C_induction}, 
we can rewrite the left hand side of \eqref{eq:C_duality} as 
\begin{align*}
\sum_{p_0=1}^p C_j(p,p_0)\frac{q}{p_0(p_0+q)}
&=\sum_{p_0=1}^p \sum_{p_1=p_0}^p C_{j-1}(p,p_1)\frac{1}{p_1^2}
\biggl(\frac{1}{p_0}-\frac{1}{p_0+q}\biggr)\\
&=\sum_{p_1=1}^p C_{j-1}(p,p_1)\frac{1}{p_1^2}
\sum_{p_0=1}^{p_1} \biggl(\frac{1}{p_0}-\frac{1}{p_0+q}\biggr). 
\end{align*}
By using the identity 
\begin{align*}
\sum_{p_0=1}^{p_1} \biggl(\frac{1}{p_0}-\frac{1}{p_0+q}\biggr)
&=\sum_{p_0=1}^\infty 
\Biggl\{\biggl(\frac{1}{p_0}-\frac{1}{p_0+q}\biggr)
-\biggl(\frac{1}{p_0+p_1}-\frac{1}{p_0+p_1+q}\biggr)\Biggr\}\\
&=\sum_{p_0=1}^q\biggl(\frac{1}{p_0}-\frac{1}{p_0+p_1}\biggr)
=\sum_{q_1=1}^q\frac{p_1}{q_1(q_1+p_1)}, 
\end{align*}
we get 
\begin{align*}
\sum_{p_1=1}^p C_{j-1}(p,p_1)\frac{1}{p_1^2}
\sum_{p_0=1}^{p_1} \biggl(\frac{1}{p_0}-\frac{1}{p_0+q}\biggr)
&=\sum_{p_1=1}^p C_{j-1}(p,p_1)\frac{1}{p_1^2}
\sum_{q_1=1}^q\frac{p_1}{q_1(q_1+p_1)}\\
&=\sum_{q_1=1}^q \frac{1}{q_1^2}\sum_{p_1=1}^p
C_{j-1}(p,p_1)\frac{q_1}{p_1(p_1+q_1)}. 
\end{align*}
Finally, using the induction hypothesis and \eqref{eq:C_induction} again, 
we obtain 
\begin{align*}
\sum_{q_1=1}^q \frac{1}{q_1^2}\sum_{p_1=1}^p
C_{j-1}(p,p_1)\frac{q_1}{p_1(p_1+q_1)}
&=\sum_{q_1=1}^q \frac{1}{q_1^2}\sum_{q_0=1}^{q_1}
C_{j-1}(q_1,q_0)\frac{p}{q_0(q_0+p)}\\
&=\sum_{q_0=1}^qC_j(q,q_0)\frac{p}{q_0(q_0+p)}. 
\end{align*}
Thus we have shown \eqref{eq:C_duality}. 
\end{proof}

\begin{proof}[Proof of Theorem \ref{thm:sec3main}]
For $k=0,\ldots,n$, put 
\begin{align*}
E(k)&=\sum_{\substack{\infty=p_0\geq\cdots\geq p_k\geq 1\\
\infty=q_{n+1}\geq q_{n}\geq\cdots\geq q_{k+1}\geq 1}}
\prod_{\alpha=1}^{k}C_{j_\alpha}(p_{\alpha-1},p_{\alpha})p_\alpha^{-e_\alpha}
\prod_{\beta=k+1}^{n}C_{j_\beta}(q_{\beta+1},q_{\beta})q_\beta^{-e_{\beta-1}}
\cdot\frac{p_k^{e_k-1}q_{k+1}}{p_k+q_{k+1}},\\
F(k)&=\sum_{\substack{\infty=p_0\geq\cdots\geq p_k\geq 1\\
\infty=q_{n+1}\geq q_{n}\geq\cdots\geq q_{k+1}\geq 1}}
\prod_{\alpha=1}^{k}C_{j_\alpha}(p_{\alpha-1},p_{\alpha})p_\alpha^{-e_\alpha}
\prod_{\beta=k+1}^{n}C_{j_\beta}(q_{\beta+1},q_{\beta})q_\beta^{-e_{\beta-1}}
\cdot\frac{p_kq_{k+1}^{e_k-1}}{p_k+q_{k+1}}. 
\end{align*}
In particular, we have $E(0)=0$ (resp.\ $F(n)=0$) because of the factor 
$\frac{q_1}{\infty+q_1}$ (resp.\ $\frac{p_n}{p_n+\infty}$) 
(recall that we set $e_0=e_n=1$). 
For general $k$, one can verify that $E(k)+F(k)=X(k)$ by using 
\[\frac{p_k^{e_k-1}q_{k+1}}{p_k+q_{k+1}}+\frac{p_kq_{k+1}^{e_k-1}}{p_k+q_{k+1}}
=\begin{cases}1 & (\text{if }e_k=1),\\ 
p_kq_{k+1} & (\text{if }e_k=3). \end{cases}\]
Moreover, when $1\leq k\leq n$, one has 
\[\sum_{p_k=1}^{p_{k-1}}C_{j_k}(p_{k-1},p_k)p_k^{-e_k}
\frac{p_k^{e_k-1}q_{k+1}}{p_k+q_{k+1}}
=\sum_{q_k=1}^{q_{k+1}}C_{j_k}(q_{k+1},q_k)q_k^{-e_{k-1}}
\frac{p_{k-1}q_k^{e_{k-1}-1}}{p_{k-1}+q_k}\]
by Lemma \ref{lem:C_duality}, hence $E(k)=F(k-1)$. 
Therefore, 
\[\sum_{k=0}^n(-1)^kX(k)=\sum_{k=0}^n(-1)^k\bigl(E(k)+F(k)\bigr)
=\sum_{k=1}^n(-1)^kF(k-1)+\sum_{k=0}^{n-1}(-1)^kF(k)=0\]
as desired. 
\end{proof}

\bigskip
Finally, we prove Theorem \ref{thm:TWIT A_n} in the introduction. 

\begin{proof}[Proof of Theorem \ref{thm:TWIT A_n}]
Since we only use $\vece=(3,1,\ldots,3,1)$ here, 
we omit it in this proof. 

For $\vecj=(j_0,\ldots,j_n)$ and $\sigma\in\sym_{n+1}$, 
we put $\sigma(\vecj):=(j_{\sigma(0)},\ldots,j_{\sigma(n)})$. 
Then the statement $(A_n)$ is expressed as 
\begin{equation}
\sum_{\sigma\in\sym_{2n}}\Zmap_{2n+1}\bigl(\sigma(\vecj)_+\bigr)
\in\Q\cdot\pi^{2m+4n}
\qquad \bigl(\forall\vecj=(j_0,\ldots,j_{2n-1})\in\Z_{\geq 0}^{2n}\bigr). 
\end{equation}
Similarly, $(B_n)$ says that 
\begin{equation}
\sum_{\sigma\in\sym_{2n+1}}\Zmap_{2n+1}\bigl(\sigma(\vecj)^+\bigr)
\in\Q\cdot\pi^{2m+4n+2}
\qquad \bigl(\forall\vecj=(j_0,\ldots,j_{2n})\in\Z_{\geq 0}^{2n+1}\bigr). 
\end{equation}

Now suppose that $(A_l)$ and $(B_l)$ hold for all $l<n$. 
For $\vecj=(j_0,\ldots,j_{2n-1})$, 
we rewrite the identity \eqref{eq:3ext} as 
\begin{align*}
\Zmap_{2n+1}(\vecj_+)+\Zmap_{2n+1}(\vecj'_+)
&=\sum_{l=1}^n 
\Zmap_{2l-1}\bigl((\vecj|_{2l-1})^+\bigr)\cdot
\Zmap_{2n-2l+1}\bigl((\vecj'|_{2n-2l+1})^+\bigr)\\
&\quad 
-\sum_{l=1}^{n-1}
\Zmap_{2l+1}\bigl((\vecj|_{2l})_+\bigr)\cdot
\Zmap_{2n-2l+1}\bigl((\vecj'|_{2n-2l})_+\bigr). 
\end{align*}
Summing up over $\sigma\in\sym_{2n}$, we obtain 
\begin{equation}\label{eq:A_n}
\begin{split}
2\sum_{\sigma\in\sym_{2n}}\Zmap_{2n+1}\bigl(\sigma(\vecj)_+\bigr)
&=\sum_{l=1}^n \sum_{\sigma\in\sym_{2n}}
\Zmap_{2l-1}\bigl((\sigma(\vecj)|_{2l-1})^+\bigr)\cdot
\Zmap_{2n-2l+1}\bigl((\sigma(\vecj)'|_{2n-2l+1})^+\bigr)\\
&\quad 
-\sum_{l=1}^{n-1} \sum_{\sigma\in\sym_{2n}}
\Zmap_{2l+1}\bigl((\sigma(\vecj)|_{2l})_+\bigr)\cdot
\Zmap_{2n-2l+1}\bigl((\sigma(\vecj)'|_{2n-2l})_+\bigr). 
\end{split}
\end{equation}

Fix an integer $l$ such that $1\leq l\leq n$, 
and take a subset $S=\{s_0,\ldots,s_{2(l-1)}\}$ of $\{0,\ldots,2n-1\}$ 
of cardinality $2l-1$. We also write 
$\{0,\ldots,2n-1\}\setminus S=\{t_0,\ldots,t_{2(n-l)}\}$, 
and put $\vec{\jmath}_1=(j_{s_0},\ldots,j_{s_{2(l-1)}})$, 
$\vec{\jmath}_2=(j_{t_0},\ldots,j_{t_{2(n-l)}})$. 
Then we have 
\begin{align}
\sum_{\substack{\sigma\in\sym_{2n}\\ \{\sigma(0),\ldots,\sigma(2(l-1))\}=S}}
&\Zmap_{2l-1}\bigl((\sigma(\vecj)|_{2l-1})^+\bigr)\cdot
\Zmap_{2n-2l+1}\bigl((\sigma(\vecj)'|_{2n-2l+1})^+\bigr)\notag\\
&=\sum_{\tau_1\in\sym_{2(l-1)+1}}
\Zmap_{2(l-1)+1}\bigl(\tau_1(\vec{\jmath}_1)^+\bigr)
\sum_{\tau_2\in\sym_{2(n-l)+1}}
\Zmap_{2(n-l)+1}\bigl(\tau_2(\vec{\jmath}_2)^+\bigr). \label{eq:B_l}
\end{align}
If we set $m_1=\sum_ij_{s_i}$ and $m_2=\sum_ij_{t_i}$, 
the assumptions $(B_{l-1})$ and $(B_{n-l})$ implies that the right hand side 
of \eqref{eq:B_l} belongs to 
$\Q\cdot\pi^{2m_1+4(l-1)+2}\cdot\pi^{2m_2+4(n-l)+2}=\Q\cdot\pi^{2m+4n}$. 
Therefore, by summing up over all $l$ and $S$, we obtain that 
\[\sum_{l=1}^n \sum_{\sigma\in\sym_{2n}}
\Zmap_{2l-1}\bigl((\sigma(\vecj)|_{2l-1})^+\bigr)\cdot
\Zmap_{2n-2l+1}\bigl((\sigma(\vecj)'|_{2n-2l+1})^+\bigr)
\in \Q\cdot\pi^{2m+4n}. \]
Similarly, $(A_l)$ for $l=1,\ldots,n-1$ imply that 
\[\sum_{l=1}^{n-1} \sum_{\sigma\in\sym_{2n}}
\Zmap_{2l+1}\bigl((\sigma(\vecj)|_{2l})_+\bigr)\cdot
\Zmap_{2n-2l+1}\bigl((\sigma(\vecj)'|_{2n-2l})_+\bigr)
\in \Q\cdot\pi^{2m+4n}. \]
Thus, from \eqref{eq:A_n}, we conclude that 
\[\sum_{\sigma\in\sym_{2n}}\Zmap_{2n+1}\bigl(\sigma(\vecj)_+\bigr)
\in \Q\cdot\pi^{2m+4n}. \]
Now the proof of Theorem \ref{thm:TWIT A_n} is complete. 
\end{proof}

\end{document}